\documentclass[12pt,a4paper]{article}
\usepackage{amsfonts}
\usepackage{mathrsfs}
\usepackage{amsmath}
\usepackage{amssymb}
\usepackage{cite}
\usepackage{relsize}
\usepackage{mathtools}
\usepackage{amsthm}
\usepackage[pdftex]{graphicx}

\usepackage[small,nohug,heads=littlevee]{diagrams}

  \DeclareMathOperator{\Aut}{Aut}
 
 \DeclareMathOperator{\im}{im}

  \DeclareMathOperator{\coker}{coker}

\newtheorem{theorem}{Theorem}
\numberwithin{theorem}{section}
\newtheorem{lemma}[theorem]{Lemma}
\newtheorem{proposition}[theorem]{Proposition}
\newtheorem{corollary}[theorem]{Corollary}
\newtheorem{conjecture}[theorem]{Conjecture}
\theoremstyle{definition}
\newtheorem{definition}[theorem]{Definition}

\newtheorem{example}[theorem]{Example}

\newcommand{\Z}{{\mathbb{Z}}}
\def\pt{*}

\def\Intersect{\bigcap}
\def\union{\cup}
\def\rdiag{\bar{\Delta}}
\def\tilDel{\widetilde{\Delta}}
\def\tilde{\widetilde}
\def\P{\mathcal{P}}

\def\smash{\wedge}

\def\intersect{\cap}
\def\isom{\cong}
\def\tensor{\otimes}
\def\dirsum{\oplus}
\def\into{\hookrightarrow}
\def\Dirsum{\bigoplus}
\def\genby#1{{\langle #1\rangle}}

\newcommand{\Ast}{\mathop{\mathlarger{\mathlarger{\ast {} }} {} }}

\def\Loop{\Omega}

\def\hat{\widehat}

\def\phi{\varphi}
\def\redH{\widetilde{H}}
\def\redh{\widetilde{b}}
\def\del{\partial}
\def\barDel{\bar{\Delta}}

\def\cat#1{\mathsf{#1}}
\def\Grp{\cat{Grp}}
\def\Grpd{\cat{Grpd}}

\def\PtSet{\cat{Set}_*}

\def\bar{\overline}

\def\Z{{\mathbb Z}}

\def\R{{\mathbb R}}
\def\F{{\mathbb{F}}}

\def\tensor{\otimes}

\def\calI{{\mathcal I}}
\def\calJ{{\mathcal J}}
\def\co{\colon\thinspace}

\newarrow{Embed}{C}{-}{-}{-}{>}
\newarrow{Equals}{=}{=}{=}{=}{=}
\newarrow{Unique}{.}{.}{.}{.}{>}

\begin{document}

\title{Homology Decompositions of the Loops on 1-Stunted Borel Constructions of $C_2$-Actions\footnote{The authors gratefully acknowledge the assistance of Singapore Ministry of Education research grants AcRF Tier 1(WBS No. R-146-000-137-112) and AcRF Tier 2 (WBS No. R-146-000-143-112). The 2nd author is supported in part by a grant (No. 11028104) of NSFC of China.}}
\author{Man Gao and Jie Wu}
\date{}
\maketitle

\begin{abstract}
Carlsson's construction is a simplicial group whose geometric realization is the loop space of the 1-stunted reduced Borel construction. Our main results are: i) Given a pointed simplicial set acted upon by the discrete cyclic group $C_2$ of order 2, if the orbit projection has a section, then this loop space has a mod 2 homology decomposition; ii) If the reduced diagonal map of the $C_2$-invariant set is homologous to zero, then the pinched sets in the above homology decomposition themselves have homology decompositions in terms of the $C_2$-invariant set and the orbit space. Result i) generalizes a previous homology decomposition of the second author for trivial actions. To illustrate these two results, we completely compute the mod 2 Betti numbers for an example.
\end{abstract}

\section{Introduction}

It is a general problem in algebraic topology to compute the homology of a loop space, failing which, to give a homology decomposition of a loop space.
In this paper, we show in Theorem \ref{thm: homology decomp of JGX} that, if the orbit projection has a section, then there is a mod 2 homology decomposition of a certain loop space $\Loop (X\rtimes_{C_2} W_\infty^1 C_2)$. This generalizes a previous homology decomposition of the second author for trivial actions (see the original paper \cite{Wu97} and Section \ref{sec: E0 is group ring} below.)

The following notational conventions that will be used throughout this paper. We reserve $G$ to denote the discrete cyclic group $C_2$ of order 2, written multiplicatively with generator $t$. In particular $t^2=1$. Let $X$ denote a pointed simplicial $G$-set. Denote by $A$ the simplicial subset of $X$ fixed under the $G$-action. Let $\F_2$ denote the finite field with two elements.

The 1-stunted reduced Borel construction $X\rtimes_{C_2} W_\infty^1 C_2$ is the homotopy cofiber of the inclusion from $X$ into its reduced Borel construction. Carlsson constructed a simplicial group $J^G[X]$ is the loop space $\Loop (X\rtimes_{C_2} W_\infty^1 C_2)$. See Section \ref{sec: preliminaries} for details.

The {\textit{orbit projection}} is the simplicial epimorphism $X\to X/G$ onto the orbit space. A {\emph{section of the orbit projection}} is a simplicial map $j\co X/G\to X$ such that the composite $X/G\xrightarrow{j} X\to X/G$ is the identity map on $X/G$. Simplicial $G$-sets whose orbit projection has a section is characterized in Proposition \ref{prop: form of decomposable Z2}.

\begin{theorem} \label{thm: homology decomp of JGX}
If the orbit projection has a section, then there an isomorphism of $\F_2$-algebras:
\begin{equation} \label{eq: E1 collapse}
\redH_*(\Omega(X\rtimes_G W_\infty^1 G);\F_2)\isom \Dirsum_{s=1}^\infty \redH_*\left((X/G)^{\smash s}/\tilDel_s;\F_2\right)
\end{equation}
Here $\tilDel_0=\tilDel_1:=\pt$ and:
\[
\tilDel_s:=\{x_1G\smash \cdots \smash x_s G\in (X/G)^{\smash s}|\, \exists i=1,\ldots, s-1\, (x_i=x_{i+1}\in A)\} \quad (s\ge 2.)
\]
\end{theorem}

To compute the direct summands in $\eqref{eq: E1 collapse}$, we can consider the long exact sequence associated to cofiber sequence $\tilDel_s\to (X/G)^{\smash s}\to (X/G)^{\smash s}/\tilDel_s$:
\begin{equation} \label{eq: les for tilDel}
\cdots\to \redH_*(\tilDel_s)\to \redH_*((X/G)^{\smash s})\to \redH_*((X/G)^{\smash s}/\tilDel_s)\to \redH_{*-1}(\tilDel_s)\to\cdots
\end{equation}
Here we suppress the coefficients $\F_2$. Our next result is a sufficient condition for there to be a homology decomposition of the pinched set $\tilDel_s$.

The reduced diagonal map of $A$ is the simplicial map $A\to A\smash A$ is given by $a\mapsto a\smash a$ for all $a\in A_n$. A pointed simplicial map $f\co Y\to Z$ is {\textit{mod 2 homologous to zero}} if the induced map $f_*\co \redH_*(Y;\F_2)\to \redH_*(Z;\F_2)$ is the zero map. We show that if the reduced diagonal map of $A$ is mod 2 homologous to zero, then the mod 2 homology of $\tilDel_s$ is completely determined by the mod 2 homology of the fixed set $A$ and the orbit space $X/G$.

The following multi-index notation is used. Let $\redh_t(Y;\F_2):=\dim \redH_t(Y;\F_2)$ denote the {\textit{$t$-th reduced mod 2 Betti number of $Y$}}, that is, the dimension of mod 2 reduced homology of a pointed simplicial set $Y$ in dimension $t$. A {\emph{multi-index}} $\alpha=(\alpha_1,\ldots,\alpha_d)$ is a (possibly empty) sequence of positive integers. The {\emph{length}} of this multi-index is $|\alpha|=\alpha_1+\cdots +\alpha_d$ and its {\emph{dimension}} is $\dim \alpha=d$. Given a multi-index $\alpha=(\alpha_1,\ldots,\alpha_d)$, write for short the following product:
\[
\redh_\alpha(Y;\F_2):=\redh_{\alpha_1}(Y;\F_2)\redh_{\alpha_2}(Y;\F_2) \cdots \redh_{\alpha_d}(Y;\F_2)
\]

\begin{theorem} \label{thm: homology decomp of tilDel}
If the reduced diagonal map of $A$ is mod $2$ homologous to zero, then the mod 2 Betti number of $\tilDel_s$ is given by:
\[
\redh_t(\tilDel_s;\F_2)\isom \sum_{\substack{|\lambda|+|\mu| =t-s+\dim\lambda+\dim\mu + 1 \\ 2\le\dim\lambda+\dim\mu+1\le s}} c_{\lambda,\mu}\redh_\lambda(X/G;\F_2) \redh_\mu (A;\F_2),
\]
where $c_{\lambda,\mu}=\binom{\dim \lambda +\dim\mu}{\dim\mu}\binom{s-\dim \lambda-\dim\mu-1}{\dim\mu-1}$.
\end{theorem}
The condition that the reduced diagonal map of $A$ is mod 2 homologous to zero is quite general. For example, this condition is satisfied if $A$ is the reduced suspension on some space (see Example \ref{eg: weak category}.)

The homology decompositions of Theorems \ref{thm: homology decomp of JGX} and \ref{thm: homology decomp of tilDel} can be applied to compute the mod 2 Betti numbers of $\Omega(X\rtimes_G W_\infty^1 G)$ for certain pointed simplicial $G$-sets $X$. These homology decompositions are particularly effective when the orbit space $X/G$ has many trivial homology groups. This leads to many zero terms appearing in the long-exact sequence \ref{eq: les for tilDel}. We illustrate with the computation of the mod 2 Betti numbers with the following example. Note that the antipodal action on the 2-sphere $S^2$ has the equatorial circle $S^1$ as the fixed set.

\begin{proposition} \label{prop: S2 S1 example}
Consider the $G$-space $S^2\union_{S^1} S^2$ formed by two 2-spheres $S^2$ with the antipodal action, with their equatorial circles identified. The reduced mod 2 Betti numbers of the loop space of its 1-truncated Borel construction is given as follows:
\begin{align*}
&\redh_{n}(\Omega([S^2\union_{S^1} S^2]\rtimes_G E_\infty^1 G)) \\
=\,\,&\begin{cases}
1+\sum_{r=k+1}^{2k} \sum_{J= 1}^{2r-3}\binom{2k-r+J}{J}\binom{2r-2k-J-1}{J-1},  &n=2k, k\ge 1, \\
\sum_{r=k+2}^{2k+1} \sum_{J= 1}^{r-k-1}\binom{2k-r+J+1}{J}\binom{2r-2k-J-2}{J-1}, &n=2k+1, k\ge 0
\end{cases}
\end{align*}
\end{proposition}

The outline of this paper is as follows. Carlsson's simplicial group construction $J^G[X]$ and the reduced 1-stunted Borel construction are introduced in Section \ref{sec: preliminaries}. In Section \ref{sec: construct E0 models}, the augmentation ideal filtration of the group ring $\F_2(J^G[X])$ is considered. We construct simplicial algebras which are isomorphic to the graded algebra associated to this filtration. Theorem \ref{thm: homology decomp of JGX} is proved in Section \ref{sec: E0 is group ring}. Theorem \ref{thm: homology decomp of tilDel} is proved in Section \ref{sec: MV ss} using the Mayer-Vietoris spectral sequence. Section \ref{sec: eg} is devoted to the example $X=S^2\union_{S^1} S^2t$ and the proof of Proposition \ref{prop: S2 S1 example}.

This paper is a revision of part of the first author's PhD thesis \cite{GaoThesis}.

\section{Preliminaries} \label{sec: preliminaries}

To begin, we explain the concepts of the reduced Borel construction and its 1-stuntation.

Denote by $WG$ any contractible simplicial set with a free $G$-action. Any two such simplicial sets are equivariantly homotopy equivalent. In our case where $G=C_2$, it is standard to give $WG$ concretely as the $\infty$-sphere $S^\infty$ with the antipodal action. Let $EG:=|WG|$ denote the geometric realization of $WG$.

The simplicial set $WG$ is filtered by simplicial $G$-subsets:
\begin{equation} \label{eq: E-filtration}
G \simeq W_0 G \subset W_1G \subset \cdots \subset W_p G \subset \cdots \subset W_\infty G =: WG
\end{equation}
Here $W_p G$ is the $p$-th skeleton of $WG$. In fact $W_p G$ is the $(p+1)$-th fold join of $G$. In our case where $G=C_2$, it is standard to give $W_pG$ concretely as the $p$-sphere $S^p$ with the antipodal action.

The {\textit{bar construction of $G$}} is the orbit space $\bar{W}G:=EG/G$. In fact $\bar{W}G$ is homotopy equivalent to the infinite-dimensional real projective space:
\[
\bar{W}G\simeq \R P^\infty
\]
The {\textit{classifying space of $G$}} is the geometric realization $BG:=|\bar{W}G|$. Since $G$ is discrete, its classifying space $BG$ is the Eilenberg-Mac Lane space $K(G,1)$.

Consider the action of $G$ on a simplicial set $X$. The free simplicial $G$-set associated to $X$ is $X\times WG$ with the diagonal action. The {\textit{Borel construction of $X$}} is the orbit space $X\times_G WG:=(X\times WG)/G$. For example, the Borel construction of the $G$-action on the standard 0-simplex $\Delta[0]=*$ is bar construction of $G$:
\[
\pt\times_G WG\simeq \bar{W}G
\]

Suppose the $G$-action is pointed, that is to say, the simplicial set $X$ has a basepoint and the $G$-action fixes the basepoint. The {\textit{reduced Borel construction}} of this pointed action, written $X\rtimes_G WG$, is the homotopy cofiber of $\pt\times_G WG\to X\times_G WG$.

More generally, let $X\times _G W_p G $ denote the orbit space $(X\times W_pG)/G$ of the diagonal action. For pointed actions, let $X\rtimes_G W_pG$ denote the homotopy cofiber of $\pt\times_G W_pG\to X\times_G W_pG.$ For $q\ge p$, define the {\textit{$(p,q)$-stunted reduced Borel construction}} $X\rtimes_G W_p^q G$ as the homotopy cofiber of $X\rtimes_G W_{q-1} G\to X\rtimes_G W_p G $. In particular, when $p=\infty$, we call $X\rtimes_G W_\infty^q G$ simply as the {\textit{$q$-stunted reduced Borel construction}} of the $G$-action on $X$.

In this paper, we are interested in the 1-stunted Borel construction $X\rtimes_G W_\infty^1 G$. Since $X\rtimes_G W_0 G \simeq X\rtimes_G G \simeq (X\times_G G)/(*\times_G G)\simeq X$, the 1-stunted Borel construction is just the homotopy cofiber of the inclusion $X\into X\rtimes_G W G$. Denote by $|X|\rtimes_G E_\infty^1 G$ the geometric realization of $X\rtimes_G W_\infty^1 G$.

Carlsson \cite{Carlsson84} constructed a simplicial group $J^G[X]$ whose geometric realization is the loop space of the 1-stunted reduced Borel construction:
\begin{equation} \label{eq: Carlsson's result}
|J^G[X]|\simeq \Loop (|X|\rtimes_G E_\infty^1 G)
\end{equation}
Carlsson's construction is given in the $n$th dimension by:
\begin{equation} \label{eq: Carlsson construction}
J^G[X]_n:=\frac{F[X_n\smash G_n]}{\genby{\forall x\in X_n \forall g,h\in G_n\, (x\smash g)\cdot (xg\smash h)\sim (x\smash gh)}}.
\end{equation}
Here $F[S]=\coker(F(\pt)\to F(S))$ is the reduced free group on a pointed set $S$, where $F(\bullet)$ denotes the (unreduced) free group. The functor $F[\bullet]\co \PtSet\to\Grp$ is the left adjoint to the inclusion functor $\Grp\into\PtSet$ that sends a group to its underlying set with the identity element as basepoint.

Carlsson's construction is the reduced universal simplicial group on the pointed simplicial action groupoid $X//G$:
\[
J^G[X]\isom U[X//G]
\]
For a pointed small groupoid $H$, its {\textit{reduced universal monoid}} $U[H]$ is defined the following cokernel:
\[
U[H] :=\coker(U(\Aut_H(\pt))\to U(H))
\]
Here $\Aut_H(\pt)$ denotes the full subcategory of $H$ whose only object is the basepoint and $H:\Grpd\to \Grp$ is the left adjoint of the inclusion functor $\Grp\into \Grpd$ that sends a group to the corresponding small groupoid with one object. The {\textit{reduced universal simplicial group}} $U[G]$ of a small simplicial groupoid $G$ is defined dimensionwise. Further details can be found in the first author's thesis \cite{GaoThesis}. This categorial viewpoint led to a unification of Carlsson's construction and a simplicial monoid construction of the second author \cite{Wu97}, which contains the classical constructions of Milnor \cite{Milnor72} and James \cite{James55} as special cases. An upcoming paper will further elaborate on this viewpoint.

\section{Augmentation Quotients as Free Simplicial Modules} \label{sec: construct E0 models}

In this Section, we construct two simplicial algebras each of which is isomorphic to the associated graded algebra of the augmentation ideal filtration of the group ring $\F_2(J^G[X])$ (see Proposition \ref{prop: models of assoc graded}.) In each dimension, each of these simplicial algebras is a quotient of a tensor algebra by a homogeneous ideal. Therefore each augmentation quotient is the reduced simplicial $\F_2$-module of a pointed simplicial set (see Corollary \ref{cor: E0 term as space}.)

In our case where $G=C_2$, there is a natural isomorphism in pointed simplicial $G$-sets $X$:
\begin{equation} \label{eq: form of Carl constru for Z2}
J^G[X]\isom \frac{F[X]}{\genby{\forall x\in X\, (x\cdot xt\sim 1)}} \\
\end{equation}
Recall from the introduction that $F[X]$ is the reduced free group on $X$. Via this natural isomorphism, we will identify $J^G[X]$ with the RHS of \eqref{eq: form of Carl constru for Z2}.

Let $K$ be a field and $H$ be a group. The elements of the group ring $K(H)$ are finite sums of the form $\sum_{\lambda\in K, h\in H} \lambda_h h$. The augmentation map $K(H)\to K$ is generated by $h\mapsto 1$ for $h\in H$. The kernel of this map is the augmentation ideal. Reserve $I$ to denote the augmentation ideal of the group ring $\F_2(J^G[X])$. The augmentation ideal $I$ is generated by $\bar{h}:=h-1$ where $h\in J^G[X]$. The powers of $I$ filter the group ring (Quillen \cite{Quillen68} calls this the $I$-filtration:)
\begin{equation} \label{eq: aug ideal filtration}
\cdots\subseteq I^{s+1}\subseteq I^s\subseteq\cdots\subseteq I^1\subseteq
I^0=\F_2(J^G[X]).
\end{equation}
We denote the spectral sequence by $\{E^r\}$ associated to this filtration. The $E^0$ term is just the graded algebra $\Dirsum_{s=0}^\infty I^{s}/ I^{s+1}$ associated to the above filtration.

One of the simplicial algebras we construct is $A^G[X]$, defined below. For an $\F_2$-module $M$, let $T(M)=\Dirsum_{s=0}^\infty \underbrace{M\tensor_{\F_2} \cdots \tensor_{\F_2} M}_s$ denote the tensor $\F_2$-algebra on $M$.
\begin{definition} \label{def: def of AGX}
Let $X$ be a pointed simplicial $G$-set. The simplicial graded $\F_2$-algebra $A^G[X]$ is defined dimensionwise by
$$(A^G[X])_n:=\frac{T(\F_2[X_n]\tensor_{\F_2(G)} \F_2)}{\genby{\forall a\in A_n\, (a\tensor_{\F_2(G)} 1)^2}}.$$
Here the $G$-action on $X_n$ allows us to view $\F_2[X_n]$ as a right $\F_2(G)$-module, while $\F_2$ is viewed as an $\F_2(G)$-module where $G$ acts trivially on the left. The tensor product $\F_2[X_n]\tensor_{\F_2(G)} \F_2$ is viewed as an $\F_2$-module.
\end{definition}

\begin{proposition}\label{prop: aug ideal is gen by x bars}
The augmentation quotient $I/I^2$ is generated by $\{\bar{x} + I^2|\, x\neq \pt\}.$
\end{proposition}

\begin{proof}
The identity $\bar{x} \thinspace\bar{y}=\bar{xy}-\bar{x}-\bar{y}$ and the fact that $\bar{x} \thinspace\bar{y}\in I^2$ implies that $(\bar{x} + I^2)+(\bar{y}+I^2)=\bar{xy}+I^2$.

The set $S:=\{\bar{x} + I^2|\, x\neq \pt\}$ thus generates
\begin{equation} \label{eq: temp eq}
\{\bar{x_1x_2\cdots x_n}+I^2|\, x_1\neq \pt,\ldots x_n\neq \pt \}
\end{equation}
By \eqref{eq: form of Carl constru for Z2}, each nonidentity element of $J^G[X]$ is (the equivalence class of) a word of the form $x_1\cdots x_n$ where all $x_1,\ldots, x_n$ are different from the basepoint $\pt$, hence $I/I^2$ is generated by \eqref{eq: temp eq} and thus also by $S$. This completes the proof.
\end{proof}

\begin{proposition} \label{prop: E^0 of completed algebra}
Let $B$ be a graded algebra. Let $\hat{B}$ be the completion of $B$ with respect to the filtration by degree:
$$\cdots\subset  B_{\ge r}\subset \cdots \subset B_{\ge 1}\subset B_{\ge 0}=B,$$
where $B_{\ge r}:=\Dirsum_{i=r}^\infty B_i$.
The above filtration induces a filtration on $\hat{B}$
$$\cdots\subset  \hat{B}_{\ge r}\subset \cdots \subset \hat{B}_{\ge 1}\subset \hat{B}_{\ge 0}=\hat{B}.$$
Then the map $\Theta\co E^0(\hat{B})\to B$ whose $r$th grade is given by $\Theta_r(f+\hat{B}_{\ge r+1})=f_r$, where $f_r$ is the $r$th degree component of $f$, is an isomorphism of graded algebras.
\end{proposition}

\begin{proof}
An element of $\hat{B}$ is a formal power series of the form $f=f_0 + f_1+ \cdots f_i + \cdots $ where $f_i$ is of degree $i$ in $B$. An element of $\hat{B}_{\ge r}$ is a formal power series of the form $f=f_r + f_{r+1} +\cdots $ whose lowest degree is at least $r$. The map $\Theta_r$ is well-defined since if $f\in\hat{B}_{\ge r+1}$, then $f_r=0$.

Let $\Lambda\co B\to E^0(\hat{B})$ be the map whose $r$th grade is $\Lambda_r(f)=f+\hat{B}_{\ge r+1}$. It is easy to check that $\Theta$ and $\Lambda$ are inverses.
\end{proof}

\begin{lemma} \label{lem: alg model of E^0 for aug ideal}
There is an isomorphism of graded algebras natural in $X$:
\begin{align*}
\Phi\co A^G[X]     &\to    E^0 \\
x\tensor_{\F_2(G)} 1 &\mapsto \bar{x} + I^2
\end{align*}
\end{lemma}

\begin{proof}
Write $\tensor:=\tensor_{\F_2(G)}$for short.

We first verify that $\Phi$ is well-defined, that is to say, that the map $(x,1)\mapsto\bar{x}+I^2$ is indeed $\F_2(G)$-linear. On the one hand, $(x\cdot t,1)\mapsto \bar{xt}+I^2$, and on the other hand $(x,t\cdot 1)=(x,1)\mapsto \bar{x}+I^2$. Since $x\cdot xt = 1$,
$$\bar{x}\bar{xt}+\bar{x}+\bar{xt}=(x-1)(xt-1)+(x-1)+(xt-1)=x\cdot xt - 1=0.$$
This implies $\bar{x}+\bar{xt}=-\bar{x}\bar{xt}\in I^2$ so that $\bar{x}+I^2=-\bar{xt}+ I^2=\bar{xt}+I^2$ as the ground field is $\F_2$. Therefore both $(x\cdot t,1)$ and $(x,t\cdot 1)$ are sent to the same thing which verifies the $\F_2(G)$-linearity of the map $(x,1)\mapsto\bar{x}+I^2$.

Our definition $\Phi(x\tensor 1)=\bar{x}+I^2$ is given for $x\in X$, then it can be extended to a map $T(\F_2[X_n]\tensor_{\F_2(G)} \F_2)\to E^0$. This is because $\Phi(\pt\tensor 1)=\bar{\pt}+I^2=1-1+I^2=0$ and the tensor algebra $T(\F_2[X_n]\tensor_{\F_2(G)} \F_2)$ is generated by elements of the form $x\tensor 1$. We check that this map factors through the defining equivalence relation of $A^G[X]$. Given $a\in A$, we have $\Phi((a\tensor 1)^2)=(\bar{a}+ I^2)^2 = \bar{a}^2+I^3$. And $\bar{a}^2=(a-1)^2=a^2-1=0$ since $a\in A$ implies $a^2=a\cdot at=1$. Thus $\Phi((a\tensor 1)^2)=\bar{a}^2+I^3=0$, so we have a well-defined map $\Phi\co A^G[X]\to E^0$.

Next we show that $\Phi$ is an epimorphism. It suffices to show that, when $\Phi$ is restricted to the first grade, the map $\Phi_1\co A^G_1[X]\to I/I^2$ is an epimorphism, since $I/I^2$ generates $E^0$. Using the isomorphism in \eqref{eq: form of Carl constru for Z2}, Proposition \ref{prop: aug ideal is gen by x bars} implies that the augmentation ideal is generated by $\bar{x}$ where $\pt\neq x\in X$. Since, for each $x\in X$, the element $\bar{x}+I^2$ is the image of $\Phi(x\tensor 1)$, this Proposition implies that each element in $I/I^2$ has a preimage under $\Phi$. Therefore $\Phi$ is an epimorphism.

To show that $\Phi$ is a monomorphism, choose a subset $B\subset X$ of elements not fixed by the $G$-action that decomposes $X$ into the disjoint union $A\sqcup B\sqcup Bt$. Then the map $f\co J^G[X]\to J^G[A]\ast F(B)$ that sends $a\mapsto a$ for $a\in A$ and $b\mapsto b, bt\mapsto b^{-1}$ for $b\in B$ is a group isomorphism. The map $e_1\co J^G[A]\to \hat{A^G[X]}$ generated by $a\mapsto a\tensor 1 + 1\tensor 1$ is well-defined. This is because $e_1(a\cdot a)= (a\tensor 1 + 1\tensor 1)(a\tensor 1 + 1\tensor 1)=(a\tensor 1)(a\tensor 1)+(1\tensor 1)(1\tensor 1)=(1\tensor 1)$ agrees with $e_1(1)=1\tensor 1$. Define $e_2\co F(B)\to \hat{A^G[X]}$ by sending $b\mapsto b\tensor 1 + 1\tensor 1$. In particular, $e_2(b^{-1})=\frac{1}{ b\tensor 1 + 1\tensor 1}=\sum_{i=0}^\infty (-1)^{i}(b\tensor 1)^i=\sum_{i=0}^\infty (b\tensor 1)^i$ as the ground field is $\F_2$. The universal property of the free product gives a map $e_1\ast e_2\co J^G[A]\ast F(B)\to \hat{A^G[X]}$. The universal property of the group ring then induces a map $\tilde{e_1\ast e_2}\co \F_2(J^G[X])\to \hat{A^G[X]}$. This induces a map $E^0(\tilde{e_1\ast e_2})\co E^0(\F_2(J^G[A]\ast F(B)))\to E^0(\hat{A^G[X]})$ between the associated graded algebras. Consider the composite
\begin{multline*}
A^G[X]\xrightarrow{\Phi}E^0(\F_2(J^G[X]))\xrightarrow{E^0(\F_2(f))}E^0(\F_2(J^G[A]\ast F(B))) \\
\xrightarrow{E^0(\tilde{e_1\ast e_2})}E^0(\hat{A^G[X]})\xrightarrow{\Theta}A^G[X].
\end{multline*}
Here $\Theta$ is the map given in Proposition \ref{prop: E^0 of completed algebra}. It is easy to check that this composite is the identity map on $A^G[X]$ and hence the first map $\Phi$ is a monomorphism, as required.

Finally, it is straightforward to the check the naturality. This completes the proof.
\end{proof}

\begin{proposition} \label{prop: models of assoc graded}
There are isomorphisms of simplicial graded $\F_2$-algebras:
\[
\Dirsum_{s=0}^\infty I^{s}/ I^{s+1} \isom A^G[X] \isom T(\F_2[X/G])/\genby{\forall a\in A\, (aG)^2}
\]
\end{proposition}
Here $T(\F_2[X/G])/\genby{\forall a\in A\, (aG)^2}$ is the simplicial graded $\F_2$-algebra whose $n$th dimension is $T(\F_2[X_n/G])/\genby{\forall a\in A_n\, (aG)^2}$.

\begin{proof}
Define the map $\Phi\co A^G[X]\to E^0(\F_2(J^G[X]))$ dimensionwise using the previous Lemma \ref{lem: alg model of E^0 for aug ideal}. In each dimension $n$, the map $\Phi_n$ is an isomorphism of graded algebras. But the naturality part of the same Lemma implies that the map $\Phi$ commutes with faces and degeneracies and hence it is a simplicial map. Therefore $\Phi$ is an isomorphism of simplicial algebras:
\begin{equation} \label{eq: temp part1 of aug ideal of E^0}
\Phi\co A^G[X]\to E^0(\F_2(J^G[X])).
\end{equation}

Denote the algebra $T(\F_2[X/G])/\genby{\forall a\in A\, (aG)^2}$ by $T$. Let $\phi\co \F_2[X]\times \F_2\to T$ send $(x, 1)\mapsto xG$. Since $\phi(x\cdot t,1)=xtG=xG$ agrees with $\phi(x,t\cdot 1)=\phi(x,1)=xG$, this map factors to a map $\F_2[X]\tensor_{\F_2(G)} \F_2\to T$ from the tensor product. The universal property of the tensor algebra defines a map $T(\F_2[X]\tensor_{\F_2(G)} \F_2)\to T$. We check that this map factors through the defining equivalence relations of $A^G[X]$. Given $a\in A$, indeed $(a\tensor 1)^2$ is sent to $(aG)^2$, which is in the quotient ideal of $T$. Thus we have a map $\tilde{\phi}\co A^G[X]\to T$.

Let $\psi\co X/G\to A^G[X]$ send $xG\mapsto x\tensor_{\F_2(G)}1$. This map $\psi$ is well-defined since $xt\tensor_{\F_2(G)} 1 = x\tensor_{\F_2(G)} t\cdot 1=x\tensor_{\F_2(G)} 1$. The universal property of the tensor algebra defines a map $\psi\co T(\F_2[X/G])\mapsto A^G[X]$. We check that this map factors through the defining equivalence relations of $T$. Given $a\in A$, indeed $(aG)^2$ is sent to $(a\tensor_{\F_2(G)}1)^2$, which is in the quotient ideal of $A^G[X]$. Thus we have a map $\tilde{\psi}\co T\to A^G[X]$.

It is easy to check that $\tilde{\phi}$ and $\tilde{\psi}$ are inverses. This gives an isomorphism
\begin{equation} \label{eq: temp part2 of aug ideal of E^0}
A^G[X]\isom T=T(\F_2[X/G])/\genby{\forall a\in A\, (aG)^2}.
\end{equation}
Combine the isomorphisms \eqref{eq: temp part1 of aug ideal of E^0} and \eqref{eq: temp part2 of aug ideal of E^0} to complete the proof.
\end{proof}

Recall from the introduction that the pointed simplicial subset $\tilDel_s$ of $(X/G)^{\smash s}$ is defined as follows. Set $\tilDel_0=\tilDel_1:=\pt$ and:
\[
\tilDel_s:=\{x_1G\smash \cdots \smash x_s G\in (X/G)^{\smash s}|\, \exists i=1,\ldots, s-1\, (x_i=x_{i+1}\in A)\} \quad (s\ge 2.)
\]

\begin{corollary} \label{cor: E0 term as space}
For $s\ge 1$, there is an isomorphism of simplicial $\F_2$-modules:
\begin{align*}
I^s/I^{s+1} &\xrightarrow{\isom} \F_2\left[(X/G)^{\smash s}/\tilDel_s\right] \\
\bar{x_1}\cdots \bar{x_s}+I^{s+1} &\mapsto x_1G\smash \cdots\smash x_sG
\end{align*}
\end{corollary}

\begin{proof}
The proof of Proposition \ref{prop: models of assoc graded} gives an isomorphism of simplicial graded algebras:
\begin{align*}
\tilde{\phi}\circ \Phi^{-1}\co \Dirsum_{s=0}^\infty I^s/I^{s+1} &\xrightarrow{\isom} T(\F_2[X/G])/\genby{\forall a\in A\, (aG)^2}\\
\bar{x} + I &\mapsto xG
\end{align*}
The $s$-th grade of this isomorphism is
\begin{align*}
I^s/I^{s+1} &\xrightarrow{\isom} \left(T(\F_2[X/G])/\genby{\forall a\in A\, (aG)^2}\right)_s\\
\bar{x_1}\cdots \bar{x_s}+I^{s+1} &\mapsto x_1G \cdots x_s G
\end{align*}

The $s$-th grade of the tensor algebra is $T_s(\F_2[X/G])$ can be identified with $\F_2[(X/G)^{\smash s}]$. Via this identification, the terms of degree $s$ in the ideal $\genby{\forall a\in A\, (aG)^2}$ are linear combinations of smash products $x_1G\smash \cdots \smash x_s G$ such that, for some $i=1,\ldots, s-1$, the elements $x_i$ and $x_{i+1}$ are equal and belong to $A$. The result follows by the definition of $\tilDel_s$.
\end{proof}

\section{Proof of Theorem \ref{thm: homology decomp of JGX}} \label{sec: E0 is group ring}

In this Section, we show that the existence of a section of the orbit projection leads to a mod 2 homology decomposition of $J^G[X]$. There are two proof ingredients. First, we show that the powers of the augmentation ideal of $\F_2(J^G[X])$ have trivial intersection. Second, we show that the exact sequences $I^{s+1}\to I^s\to I^s/I^{s+1}$ are split. These imply the $\F_2(J^G[X])$ is isomorphic to $E^0$ and that the long exact sequence associated to $I^{s+1}\to I^s\to I^s/I^{s+1}$ splits into short exact sequences. Therefore the spectral sequence associated to the augmentation ideal filtration collapses at the $E^1$ term and converges to $H_*(J^G[X];\F_2)$.

We begin with a characterization of the $G$-sets whose orbit projection has a section.
\begin{proposition} \label{prop: form of decomposable Z2}
The orbit projection has a section if and only if, there exist simplicial sets $A$ and $Y$ with $A$ as a simplicial subset of $Y$, such that $X$ is a pushout $Y\union_A Yt$ with the action of flipping $Y$ with $Yt$.
\end{proposition}

\begin{proof}
If $j$ is a section of the orbit projection, then $X=\im j\union_A (\im j) t$ where $A\subset X$ is the set fixed under the action.

Conversely, the orbit space of a pushout $Y\union_A Yt$ is isomorphic to $Y$. Thus the map $Y\into Y\union_A Yt$ that is the inclusion to the left copy of $Y$ gives the required section.
\end{proof}
For the $G$-set $Y\union_A Yt$, its orbit space is isomorphic to $Y$ and the set fixed under the action is just $A$. There are two sections of the orbit projection. One section maps the orbits space to $Y$, the other section maps the orbit space to $Yt$.

In the case where the coefficient ring is a field, there is a characterization of group rings for which the powers of the augmentation ideal to have trivial intersection. We recall below the characterization if the coefficient ring is a field of prime characteristic (see Theorem 2.26 of \cite{Passi79}.)

We use the following terminology from group theory. A group has {\emph{bounded exponent}} if there exists an integer $n\ge 0$ such that every element of the group has order at most $n$. We say $\P$ is a {\textit{property of groups}} if (i) the trivial group has the property $\P$ and (ii) given isomorphic groups $G$ and $H$, the group $G$ has property $\P$ if and only if the group $H$ has property $\P$. A group $G$ is residually $\P$ if, for each nonidentity element $x\in G$, there exists a group epimorphism $\phi\co G\to H$ where $H$ is a $\P$-group such that $\phi(x)\neq 1$.

\begin{proposition} \label{prop: from passi}
Let $J$ be the augmentation ideal of a group ring $K(H)$ where $K$ is a field of characteristic prime $p$. Then $\Intersect_n J^n=0$ if and only if $H$ is residually nilpotent $p$-group of bounded exponent.
\end{proposition}

We will need the following result of Gruenberg \cite{Gruenberg57}.

\begin{lemma} \label{lem: Gruenberg result}
The free product of finitely many residually finite $p$-groups is a residually finite $p$-group.
\end{lemma}

Let $C_\infty$ denote the infinite cyclic group and $C_p$ denote the cyclic group of order $p$.

\begin{proposition}\label{prop: Mark Sapir result}
A free product of arbitrarily many copies of $C_\infty$'s and $C_p$'s is a residually finite $p$-group.
\end{proposition}

\begin{proof}
Let a group $G$ which is a free product of $C_\infty$'s and $C_p$'s be given. We write $G=\Ast_{i\in I} H_i$ where $I$ is an index set and $H_i$ is an isomorphic copy of either $C_\infty$ or $C_p$. For each $i\in I$, fix a generator $t_i$ of $H_i$.

Let a word $w=t_{i_1}^{n_{i_1}}\cdots t_{i_k}^{n_{i_k}}$ be given. Let $H= H_{i_1}\ast \cdots \ast H_{i_k}$. Let $\psi\co G\to H$ be the group homomorphism given by
\[
\psi(t_j) = \begin{cases}
t_j, & {\text{if }} j=i_1,\ldots, i_k\\
1_H, & {\text{otherwise.}}
\end{cases}
\]
This $\psi(w)$ is a nonidentity element of $H$.

It is easy to show that $C_p$ and $C_\infty$ are both residually finite $p$-groups. Thus Lemma \ref{lem: Gruenberg result} implies that the group $H$ is a residually finite $p$-group. Since $\psi(w)$ is a nonidentity element of $H$, there exists a group epimorphism $\phi\co H\to K$ where $K$ is a finite $p$-group such that $\phi(\psi(w))\neq 1$. Since the composite $G\xrightarrow{\psi} H \xrightarrow{\phi} K$, this proves that $G$ a residually finite $p$-group.
\end{proof}

The following proposition is straightforward and its proof is omitted.

\begin{proposition} \label{prop: decomposition of JGX}
Let $X$ be a pointed $G$-set. If $X$ is written as a disjoint union $A\sqcup B\sqcup Bt$, then there is a group isomorphism
\begin{align*}
J^G[X] &\to J^G[A] \ast F(B) \\
a      &\mapsto a \\
b      & \mapsto b
\end{align*}
\end{proposition}

In particular $\phi(bt)=\phi(b)^{-1}=b^{-1}$.

\begin{corollary} \label{cor: resdidual nil of gpring}
Let $X$ be a pointed $G$-set. The augmentation ideal $I$ of $\F_2(J^G[X])$ satisfies $\Intersect_{n} I^n=0$.
\end{corollary}

\begin{proof}
Write $X$ as a disjoint union $A\sqcup B\sqcup Bt$, then Proposition \ref{prop: decomposition of JGX} gives an isomorphism $J^G[X]\isom J^G[A] \ast F(B)$. The group $J^G[A]$ is a free product of $C_2$'s while the free group $F(B)$ is a free product of $C_\infty$'s. Proposition \ref{prop: Mark Sapir result} applies to show that $J^G[X]$ is a residually finite $2$-group. Since a finite $2$-group is a nilpotent $2$-group of bounded exponent, the group $J^G[X]$ is a residually nilpotent $2$-group of bounded exponent. Then the result follows from Proposition \ref{prop: from passi}.
\end{proof}

This corollary implies that the spectral sequence $\{E^r\}$ is weakly convergent.

\begin{proposition} \label{prop: gp ring as EO}
Let $J$ be the augmentation ideal of its group ring $K(H)$ with coefficients in a field $K$. If $\bigcap_n J^n=0$ and the short exact sequence $J^{s+1}\to J^s\to J^s/J^{s+1}$ is split for all $s$, then there is an isomorphism of $K$-modules:
\[
K(H) \isom \Dirsum_{s=0}^\infty J^s/J^{s+1}.
\]
\end{proposition}

\begin{proof}
Since the coefficients are taken in a field, the split short exact sequences imply that $J^s\isom J^{s+1}\dirsum J^s/J^{s+1}$ for all $s$. An easy induction shows that $K(H)\isom J^{n} \dirsum \Dirsum_{s=0}^{n-1} J^s/J^{s+1}$ for all $n$. Thus there is an isomorphism of $K$-modules for each $n$:
\[
\Dirsum_{s=0}^{n-1} J^s/J^{s+1} \isom K(H)/J^n.
\]
This allows us to identify the filtered system
\[
K(H)/J^1 \to \Dirsum_{s=0}^{1} J^s/J^{s+1}\to\cdots \to \Dirsum_{s=0}^{n-1} J^s/J^{s+1} \to \cdots
\]
with the filtered system
\[
K(H)/J^1 \to K(H)/J^2\to\cdots \to K(H)/J^n \to \cdots
\]
Therefore the colimits are isomorphic as $K$-modules:
\begin{align*}
\Dirsum_{s=0}^\infty J^s/J^{s+1} &\isom \varinjlim_n \Dirsum_{s=0}^{n-1} J^s/J^{s+1} \\
&\isom \varinjlim_n K(H)/J^n \\
&\isom K(H)/\bigcap_n J^n \\
&= K(H),
\end{align*}
where we used the assumption that $\bigcap_n J^n$ is trivial is the last step.
\end{proof}

\begin{proof}[Proof of Theorem \ref{thm: homology decomp of JGX}]
First we show that the following short exact sequence is split for each $s$:
\begin{equation} \label{eq: to split or not to split}
I^{s+1}\to I^s\to I^s/I^{s+1}
\end{equation}
For $s=0$, the short exact sequence \eqref{eq: to split or not to split} always splits for any group ring. For $s\ge 1$, Corollary \ref{cor: E0 term as space} gives an isomorphism $I^s/I^{s+1}\to \F_2\left[(X/G)^{\smash s}/\tilDel_s\right]$ defined by $\bar{x_1}\cdots \bar{x_s}+I^{s+1}\mapsto x_1G\smash \cdots\smash x_sG$. Via this isomorphism, it suffices to show that the following map has a section:
\begin{align*}
\alpha \co I^s &\to \F_2\left[(X/G)^{\smash s}/\tilDel_s\right] \\
      \bar{x_1}\cdots \bar{x_s}   &\mapsto x_1G\smash \cdots\smash x_sG.
\end{align*}

By Proposition \ref{prop: form of decomposable Z2}, the assumption that the orbit projection has a section allows us to write $X=Y\union_A Yt$. Then every orbit is of the form $y G$ for some $y\in Y$. Define $\beta\co \F_2\left[(X/G)^{\smash s}/\tilDel_s\right]\to I^s$ by $\beta(y_1G\smash \cdots y_s G)=\bar{y_1}\cdots \bar{y_s}$ for $y_1,\ldots, y_s\in Y$. The map $\beta$ is well-defined since if there exists some $i=1,\ldots, s-1$ such that both $y_i$ and $y_{i+1}$ are equal to some $a\in A$, then $\bar{y_i}\,\bar{y_{i+1}}=(a-1)(a-1)=a^2 - 1 = 1-1=0$ as $a^2=1$ in $J^G[X]$ so that $\beta(y_1G\smash \cdots \smash y_s G)=0$. Then $\beta$ is a section of $\alpha$:
$$\alpha(\beta(y_1 G\smash \cdots\smash y_s G))=\alpha(\bar{y_1}\cdots \bar{y_s})=y_1 G\smash \cdots\smash y_s G.$$
Thus we have shown that the exact sequences \eqref{eq: to split or not to split} are split for each $s$.

We have shown that $\Intersect_n I^n=0$ in Corollary \ref{cor: resdidual nil of gpring}. Thus Proposition \ref{prop: gp ring as EO} implies
\begin{equation} \label{eq: group ring is E0}
\F_2(J^G[X])\isom \Dirsum_{s=0}^{\infty} I^s/I^{s+1} \isom \F_2 \dirsum \Dirsum_{s=1}^{\infty} I^s/I^{s+1}
\end{equation}
Using Corollary \ref{cor: E0 term as space} and taking homotopy gives
\[
\pi_*(\F_2(J^G[X]))\isom \pi_*(\F_2)\dirsum\Dirsum_{s=1}^\infty \pi_*\left(\F_2\left[(X/G)^{\smash s}/\tilDel_s\right]\right)
\]
Using the Dold-Thom theorem, this becomes
\[
H_t(J^G[X];\F_2) \isom \begin{cases}
\F_2\dirsum \Dirsum_{s=1}^\infty \redH_0 \left((X/G)^{\smash s}/\tilDel_s;\F_2\right), & {\text{if }} t=0 \\
 \Dirsum_{s=1}^\infty \redH_t \left((X/G)^{\smash s}/\tilDel_s;\F_2\right), & {\text{otherwise}} \\
\end{cases}
\]
Thus the reduced homology of $J^G[X]$ is
\[
\redH_*(J^G[X];\F_2)\isom \Dirsum_{s=1}^\infty \redH_*\left((X/G)^{\smash s}/\tilDel_s;\F_2\right)
\]
The homotopy equivalence \eqref{eq: Carlsson's result} completes the proof.
\end{proof}
Note that the splitting of the short exact sequence \eqref{eq: to split or not to split} implies that the associated long exact sequence in homology splits into short exact sequences. Thus the spectral sequence $\{E^r\}$ collapses at the $E^1$ term. The isomorphism \eqref{eq: group ring is E0} between $\F_2(J^G[X])$ and $E^0$ implies that this spectral sequence converges to $H_*(J^G[X];\F_2)$.

Theorem \ref{thm: homology decomp of JGX} should be compared with the following result of the second author.

\begin{proposition}[Theorem 1.1 in \cite{Wu97}]
Let $F=\R,\mathbb{C}$ or $\mathbb{H}$ and let $X$ be a pointed space. Suppose that $H_*$ is a multiplicative homology theory such that (1) both $\bar{H}_*(FP^\infty)$ and $\bar{H}_*(FP^\infty_2)$ are free $H_*(pt)$-modules; and (2) the inclusion of the bottom cell $S^d\to FP^\infty$ induces a monomorphism in the homology. Then there is a multiplicative filtration $\{F_r H_*\Loop (FP^\infty\smash X)|\, r\ge 0\}$ of $H_*\Loop (FP^\infty \smash X)$ such that $F_0=H_*(pt)$ and
$$F_s/F_{s-1}\isom \Sigma^{(d-1)s} \bar{H}_*(X^{\smash s}/\widehat{\Delta}_s)$$
where $d=\dim_{\R} F$, $\Sigma$ is the suspension, $\widehat{\Delta}_1=\pt$ and $\widehat{\Delta}_s=\{x_1\smash\cdots\smash x_s\in X^{\smash s}|\ x_i=x_{i+1}\ for\ some\ i\}$ for $s>1$. Furthermore, this filtration is natural with respect to $X$.
\end{proposition}

Take $F=\R$. In this case, the above result holds for the reduced mod 2 homology. Since $\F_2$ is a field, the multiplicative filtration yields the homology decomposition:
\begin{equation} \label{eq: Wu's decomposition}
\redH_*(\Loop (\R P^\infty\smash X);\F_2) = \Dirsum_{s=0}^\infty \redH_*(X^{\smash s}/\widehat{\Delta}_s;\F_2)
\end{equation}
If $G=C_2$ acts on $X$ trivially, then $X$ coincides with its orbit space $X/G$. This induces an isomorphism of simplicial sets for each $r$:
\[
(X/G)^{\smash s}/\tilDel_s \isom X^{\smash s}/\widehat{\Delta}_s
\]
The 1-stunted reduced Borel construction has the following geometric realization for the trivial action:
\[
|X\rtimes_G W_\infty^1 G|\simeq \R P^\infty\smash X
\]
Therefore our homology decomposition in Theorem \ref{thm: homology decomp of JGX} generalizes \eqref{eq: Wu's decomposition}.

\section{Proof of Theorem \ref{thm: homology decomp of tilDel}} \label{sec: MV ss}

We have shown in the previous section that, if the orbit projection has a section, then $\redH_*(J^G[X];\F_2) \isom \Dirsum\redH_*\left((X/G)^{\smash s}/\tilDel_s;\F_2\right)$. The pinched set $\tilDel_s$ can be written as the following union (see Corollary \ref{cor: decomposition of tilDel}):
\begin{equation} \label{eq: decomposition of tilDel}
\left(\rdiag(A)\smash (X/G)^{\smash s-2}\right) \union \left((X/G)\smash \rdiag(A)\smash (X/G)^{\smash s-3}\right)\union \cdots \union
\left((X/G)^{\smash s-2} \smash \rdiag(A)\right).
\end{equation}
Here $\rdiag(A):=\{aG\smash aG|\, a\in A\}\subset (X/G)^{\smash 2}$.

Given a pointed simplicial set $Y$ written as a union $Y_1\union\cdots\union Y_N$ of pointed simplicial subsets, the Mayer-Vietoris spectral sequence allows one to approximate the homology of $Y$ in terms of the homology of the intersections of the $Y_i$'s. Expression \eqref{eq: decomposition of tilDel} suggests using the Mayer-Vietoris spectral sequence to study the homology of $\tilDel_s$. This can be combined with Theorem \ref{thm: homology decomp of JGX} to obtain further information about the mod 2 homology of $J^G[X]$. We illustrate this in Proposition \ref{prop: S2 S1 example}.

We briefly review the Mayer-Vietoris spectral sequence. References for this spectral sequence are \cite{Cai11,CLW11,Hatcher04}. Suppose that $Y=Y_1\union\cdots \union Y_N$ is a pointed simplicial set with each $Y_i$ a pointed simplicial subset of $Y$. Associated with $Y$ is an abstract
simplicial complex $K$ with vertices $1, 2,\dots, N$ and $\{i_1,\dots,
i_p\}\in K$ for $Y_{i_1}\intersect \cdots \intersect Y_{i_p}$. For each $\calI=\{i_1,\dots, i_p\}\in K$, define $Y_\calI=Y_{i_1}\intersect \cdots \intersect Y_{i_p}.$ In particular $Y_\emptyset =Y$.

For any simplicial set $W$, let $\Z W$ denote the free simplicial abelian group on $W$. One has a chain complex $(\Z W,\del)$:
$$\Z W_0\xleftarrow{\del} \Z W_1 \xleftarrow{\del} \Z W_2\xleftarrow{\del}\cdots,$$ where $\del=\sum_{i=0}^n (-1)^i d_i$ and $d_i$ is the $i$th face of the simplicial abelian group $\Z W$. The homology of this chain complex is the integral homology of $W$ (see Page 5 in \cite{GeorssJardine99}):
\[
H_*(W;\Z)\isom H_*(\Z W,\del).
\]
If $W$ is pointed, its mod 2 reduced homology of $W$ is given by:
$$\redH_*(W; \F_2)= \coker\left(H_*(\Z\! \pt\tensor \F_2,\del) \to H_*(\Z W\tensor \F_2,\del)\right).$$

Let $E_{p,q}=\Dirsum_{\# \calI=p} (\Z Y_{\calI}\tensor \F_2)_q$ where $\# \calI$ denotes the number of elements in the set $\calI$. Then $E=\Dirsum_{p,q} E_{pq}$ is a double complex. For $\alpha_q^\calI\in (\Z Y_{\calI}\tensor \F_2)_q$, the vertical differential is $\del^{v}(\alpha_q^\calI):=\del\alpha_q^\calI$, which is the above differential of the chain complex $\Z Y_{\calI}\tensor \F_2$. For $\alpha_q^\calI\in (\Z Y_{\calI}\tensor \F_2)_q$ where $\calI=\{i_1,\ldots, i_p\}$, the horizontal differential is $\del^h(\alpha_q^\calI):=\alpha_q^{\del \calI}:=\sum_{j=1}^p (-1)^j \alpha_q^{\partial_j \calI}$ where $\partial_j \calI:=(i_1,\ldots, \hat{i_j},\ldots, i_p)$ has $p-1$ elements by omitting the $j$th term. Here $\alpha_q^{\partial_j \calI}$ is an element of $(\Z Y_{\partial_j \calI}\tensor \F_2)_q$ via the inclusion $Y_{\calI}\into Y_{\partial_j \calI}$.

Write $E_p=\Dirsum_q E_{p,q}$. The homology of $E_0$ is the mod 2 homology of $Y$ (see \cite{Hatcher04}):
$$\redH_*(Y;\F_2)\isom H_*(E_0).$$
There is an exact sequence (see Page 94 in \cite{BottTu82}):
$$0\to E_{N}\xrightarrow{\partial_N^h} \cdots \xrightarrow{\partial_1^h} E_0 \to 0.$$

Denote $F_0=\im \partial_1^h, \ldots, F_{N-2}=\im \partial_{N-1}^h, F_{N-1}=\im \partial_{N}^h$. Then we have the short exact sequences
\begin{align*}
0\to E_{N} &\to F_{N-1} \to 0 \\
0\to F_{N-1} &\to E_{N-1} \to F_{N-2} \to 0 \\
0\to F_{N-2} &\to E_{N-2} \to F_{N-3} \to 0 \\
&\vdots\\
0\to F_{1} &\to E_{1} \to F_{0} \to 0.
\end{align*}
With respect to the differential $\partial^v\co E_{p,q}\to E_{p,q-1}$, we obtain long exact sequences
\begin{align*}
\cdots \rightarrow H_q(F_{N-2}) \xrightarrow{i} H_q(E_{N-2}) \xrightarrow{j} H_q (F_{N-3}) \xrightarrow{\zeta} H_{q-1}(F_{N-2}) \to \cdots \\
\vdots \\
\cdots \rightarrow H_q(F_1) \xrightarrow{i} H_q(E_1) \xrightarrow{j} H_q (F_0) \xrightarrow{\zeta} H_{q-1}(F_1) \to \cdots
\end{align*}
This long exact sequence can be written as an exact couple where $i$ has bidegree $(0,0)$, $j$ has bidegree $(0,-1)$ and $\zeta$ has bidegree $(-1,1)$:
\begin{diagram}
H_*(F_*) & \rTo^{(0,0)}_{i}                  & H_*(E_*) \\
         & \luTo(1,2)_{(-1,1)}^{\zeta} \ldTo(1,2)^{j}_{(0,-1)} &     \\
         & H_*(F_*).
\end{diagram}

The resulting spectral sequence is the Mayer-Vietoris spectral sequence
$$\{E^r_{p,q}(X_1\union\cdots \union X_N), d^r\}\Rightarrow H_{p+q-1}(E_0)=\redH_{p+q-1}(X;\F_2),$$
where the $r$th differential $d^r\co E^r_{p,q}\to E^r_{p-r,q+r-1}$ is induced by $i\circ \zeta^{-r+1}\circ j$ for $r\geq1$. Note $H_t(E_0)=\bigoplus_{p+q-1=t}E^\infty_{p,q}$. The $E^1$ term of this spectral sequence is
\begin{equation} \label{eq: generic E1 term}
E^1=\Dirsum_{p+q-1=t}\Dirsum_{\substack{X_\calI\neq \emptyset \\ \# \calI=p\ge 1}} \redH_q(X_\calI;\F_2).
\end{equation}

For the rest of this paper, we write $\redH(\bullet)$ as $\redH(\bullet;\F_2)$ for short. Recall from the introduction that the pointed simplicial subset $\tilDel_s$ of $(X/G)^{\smash s}$ is defined as follows. Set $\tilDel_0=\tilDel_1:=\pt$ and:
\[
\tilDel_s:=\{x_1G\smash \cdots \smash x_s G\in (X/G)^{\smash s}|\, \exists i=1,\ldots, s-1\, (x_i=x_{i+1}\in A)\} \quad (s\ge 2.)
\]
These simplicial sets $\tilDel_s$ have the following alternative inductive definition.

\begin{proposition} \label{prop: ind def of tilDel}
The simplicial sets $\tilDel_s$ can be defined inductively by:
\begin{align*}
\tilDel_0 &= \tilDel_1=\pt, \\
\tilDel_2 &= \rdiag(A), \\
\tilDel_s &= \left(\tilDel_{s-1}\smash (X/G) \right) \union \left((X/G)^{\smash s-2} \smash \rdiag(A)\right), \quad s\ge 3.
\end{align*}
\end{proposition}

\begin{proof}
We have $\tilDel_0 = \tilDel_1 = \pt$ by definition. It is easy to check that $\tilDel_2 = \rdiag(A)$. We will show that
$$\tilDel_s = \left(\tilDel_{s-1}\smash (X/G)\right) \union \left((X/G)^{\smash s-1} \smash \rdiag(A)\right).$$
Let an element $x_1G\smash \cdots \smash x_s G$ of $\tilDel_s$ be given. There are two cases: either $x_{s-1}=x_s\in A$ or $x_i=x_{i+1}\in A$ for some $1\le i<s-1$. In the former case $x_1G\smash \cdots \smash x_s G$ belongs to $(X/G)^{\smash s-1} \smash \rdiag(A)$. In the latter case $x_1G\smash \cdots \smash x_s G$ belongs to $\tilDel_{s-1}\smash (X/G)$. Hence in either case $x_1G\smash \cdots \smash x_s G$ belongs to the union $\tilDel_{s-1}\smash (X/G) \union (X/G)^{\smash s-1} \smash \rdiag(A)$. This proves one inclusion.

The proof of the reverse inclusion is similar.
\end{proof}

\begin{corollary} \label{cor: decomposition of tilDel}
For $s\ge 2$, the simplicial set $\tilDel_s$ decomposes into the following union:
$$\left(\rdiag(A)\smash (X/G)^{\smash s-2}\right) \union \left((X/G)\smash \rdiag(A)\smash (X/G)^{\smash s-3}\right)\union \cdots\union
\left((X/G)^{\smash s-2} \smash \rdiag(A)\right).$$
\end{corollary}

Before we prove this corollary, we introduce multi-index notation to abbreviate the expressions. Recall from the introduction that a multi-index $\alpha=(\alpha_1,\ldots, \alpha_d)$ is a (possibly empty) sequence of positive integers.

\begin{definition}
For $k\ge 2$, let $\barDel^{k}(A)$ denote the pointed simplicial subset of $(X/G)^{\smash k}$ whose elements are $\underbrace{aG\smash \cdots \smash aG}_k$ for some $a\in A$. We set $\barDel^{1}(A):=X/G$. For a multi-index $\alpha=(\alpha_1,\ldots, \alpha_d)$, denote by $\barDel^\alpha$ the pointed simplicial set $\barDel^{\alpha_1}(A)\smash \cdots \smash \barDel^{\alpha_d}(A)$.
\end{definition}

The pointed simplicial set $\barDel^{\alpha}$ is a subset of $$(X/G)^{\smash \alpha_1} \smash \cdots \smash (X/G)^{\smash \alpha_d}=(X/G)^{\smash \alpha_1+\cdots +\alpha_d}=(X/G)^{\smash |\alpha|}.$$

\begin{proof}[Proof of Corollary \ref{cor: decomposition of tilDel}]
We proceed by induction. For $s=2$, the RHS reduces to $\rdiag(A)$. Indeed Proposition \ref{prop: ind def of tilDel} says that $\tilDel_2= \rdiag(A)$.

Suppose the above decomposition holds for some $s\ge 2$. In the shorthand notation, the induction hypothesis becomes
$$\tilDel_s=\barDel^{(2,1,\ldots, 1)}\union \barDel^{(1,2,\ldots, 1)}\union\cdots\union \barDel^{(1,1\ldots, 2)},$$
where all the multi-indices are of length $s$. Thus
\begin{align*}
\tilDel_{s+1} &= \left(\tilDel_{s}\smash (X/G) \right) \union \left((X/G)^{\smash s} \smash \rdiag(A)\right) \\
          &= \left[\left(\barDel^{(2,1,\ldots, 1)}\union \barDel^{(1,2,\ldots, 1)}\union\cdots\union \barDel^{(1,1\ldots, 2)}\right)\smash (X/G)\right] \union \barDel^{(1,1,\ldots, 1,2)} \\
          &= \barDel^{(2,1,\ldots, 1,1)}\union \barDel^{(1,2,\ldots, 1,1)}\union \cdots \union\barDel^{(1,1,\ldots, 2,1)}\union\barDel^{(1,1,\ldots, 1,2)}.
\end{align*}
Then we can use $\barDel^{\beta}\smash (X/G)=\barDel^{(\beta,1)}$ in the last line. This proves the induction step.
\end{proof}

Recall from the introduction that a pointed simplicial map $f\co Y\to Z$ is {\textit{mod 2 homologous to zero}} if the induced map on homology $\redH_*(Y;\F_2)\to \redH_*(Z;\F_2)$ is the zero map. Since we are using homology with coefficients in $\F_2$ throughout, we throw out the reference to \lq\lq mod $2$\rq\rq. Let $f_1,\ldots, f_k$ be pointed simplicial maps. If $f_i$ is homologous to zero for some $i=1,\ldots, k$, then the smash product $f_1\smash \cdots \smash f_k$ is homologous to zero. This is because the induced map $(f_1\smash \cdots \smash f_k)_*$ on homology is just the tensor product $(f_1)_*\tensor \cdots \tensor (f_k)_*$.

\begin{proposition} \label{prop: inc are homologous to zero}
Let $\alpha$ and $\beta$ be multi-indices of length $s$. Suppose that the reduced diagonal map of $A$ is homologous to zero. If $\barDel^\alpha$ is a proper subset of $\barDel^\beta$, then the inclusion $\barDel^\alpha\into\barDel^\beta$ is homologous to zero.
\end{proposition}

\begin{proof}
The higher reduced diagonal map $d_k\co A\to A^{\smash k}$ is given by $a\mapsto\underbrace{a\smash\cdots\smash a}_k$ for $a\in A_n.$
We first show that for $k\ge 2$, the higher reduced diagonal map $d_k\co A\to A^{\smash k}$ is homologous to zero. This map is a monomorphism with image $\barDel^k(A)\isom A$. We can write $d_k$ as a composite:
$$d_k:\co A\to A\smash A \xrightarrow{1_A \smash d_{k-1}}A\smash \barDel^{k-1}(A)\into A^{\smash k}.$$
Since the first map is homologous to zero by assumption, $d_k$ is homologous to zero.

Now we return to the proposition. First consider the case where $\barDel^\alpha$ is just $\barDel^s(A)$, that is the case where $\dim\alpha=1$. Since $\barDel^\alpha$ is a proper subset of $\barDel^\beta$ by assumption, $e:=\dim\beta \ge 2$. There is a commutative diagram
\begin{diagram}
\barDel^s(A)                    & \rEmbed                   & \barDel^\beta \\
\uTo<{d_s}>{\isom}              &                           & \uTo>{d_{\beta_1}\smash \cdots \smash d_{\beta_{e}}} \\
A                               & \rTo_{d_e}        & \underbrace{A\smash \cdots \smash A}_{e}.
\end{diagram}
Since $e\ge 2$, the reduced diagonal map $d_e$ is homologous to zero from what we have shown above. Since $A\xrightarrow{d_s}\barDel^s(A)$ is an isomorphism, the inclusion $\barDel^s(A)\into \barDel^\beta$ is homologous to zero.

Finally we prove the general case where $\dim\alpha=d >1$. Since $\barDel^\alpha$ is a proper subset of $\barDel^\beta$ by assumption, we can decompose the multi-index $\beta$ into $\beta=(\gamma^{(1)},\ldots, \gamma^{(d)})$ such that $\alpha_1=|\gamma^{(1)}|,\ldots, \alpha_d=|\gamma^{(d)}|$. Thus the inclusion map $\barDel^\alpha\into\barDel^\beta$ decomposes into a smash product of the inclusions $\barDel^{\alpha_j}(A)\into \barDel^{\gamma^{(j)}}$:
\begin{diagram}
\barDel^\alpha & \rEmbed & \barDel^\beta\\
\dEquals       &         & \dEquals      \\
\barDel^{\alpha_1}(A)\smash \cdots \smash \barDel^{\alpha_d}(A) &\rTo &\barDel^{\gamma^{(1)}}\smash \cdots\smash \barDel^{\gamma^{(d)}}.
\end{diagram}
Each inclusion $\barDel^{\alpha_j}(A)\into \barDel^{\gamma^{(j)}}$ reduces to the case above. Hence it is homologous to zero. Thus after taking the smash product, the inclusion $\barDel^\alpha\into\barDel^\beta$ is homologous to zero. (Actually one $j$ is enough.)
\end{proof}

\begin{example} \label{eg: weak category}
If $A=\Sigma Y$, then the reduced diagonal map of $A$ is null-homotopic and thus homologous to zero. The {\emph{weak category}} of a space $A$ is the least $k$ such that the higher reduced diagonal $A\to \underbrace{A\smash \cdots \smash A}_k$ is null-homotopic (see Definition 2.2 of \cite{BersteinHilton60}). For example, a non-contractible suspension space has weak category 2. Berstein and Hilton \cite{BersteinHilton60} introduced the notion of weak category to study the Lusternik-Schnirelmann category. The {\textit{Lusternik-Schnirelmann category}} of a topological space $X$ is the smallest integer number $k$ such that there is an open covering $\{U_i\}_{1\leq i\leq k}$ of $X$ with the property that each inclusion map $U_i\hookrightarrow X$ is null-homotopic (see \cite{Fox39} and the references therein.)
\end{example}

\begin{proposition} \label{prop: intersections as smash prod}
The collections $$\{\barDel_\calI|\, \calI\subset \{(2,1,\ldots, 1),\ldots, (1,1\ldots, 2)\}, \# \calI =p\}$$ and $\{\barDel^\alpha|\, \dim\alpha=s-p\}$ are equal for $p=1,\ldots, s-1$. Here all the multi-indices are of length $s$.
\end{proposition}
\begin{proof}
Recall that $\barDel_\calI=\Intersect_{\alpha\in \calI} \barDel^\alpha$. We proceed by induction on $p$. The base step $p=1$ is obvious.

Let $\calI=\{\gamma^{(j_1)},\ldots, \gamma^{(j_p)}\}$ where $j_1<\cdots < j_p$ and $\gamma^{(j)}$ is the multi-index $(1,\ldots, 2,\ldots, 1)$ with $2$ as the $j$th entry. By the inductive hypothesis, there exists some $\beta$ of dimension $s-(p-1)=s-p+1$ such that $\barDel_\calJ=\barDel^\beta$ where $\calJ=\{\gamma^{(j_1)},\ldots, \gamma^{(j_{p-1})}\}$. Recall that $\barDel_\calJ=\barDel^{\gamma^{(j_1)}}\intersect\cdots\intersect\barDel^{\gamma^{(j_{p-1})}}$. Since $j_{p-1}$ is the largest term in this intersection, we can decompose $\beta$ into $(\beta',\underbrace{1,\ldots, 1}_{s-j_{p-1}-1})$. Since $\beta$ has dimension $s-p+1$, $\beta'$ has dimension $$(s-p+1)-(s-j_{p-1}-1)=j_{p-1}-p+2.$$ There are two cases: either $j_{p}=j_{p-1}+1$ or $j_{p}>j_{p-1}+1$.

Consider the case where $j_p=j_{p-1}+1$. Then $\barDel_\calI=\barDel^\delta$ where, writing $e:=\dim\beta'$,
$$\delta=(\beta'_1,\ldots,\beta'_{e-1},\beta'_e + 1,\underbrace{1,\ldots, 1}_{s-j_{p-1} - 2}).$$
Then $$\dim\delta = e + (s-j_{p-1} - 2)=(j_{p-1}-p+2)+(s-j_{p-1} - 2)=s-p,$$ which proves the induction step for this case.

Next consider the case where $j_{p}>j_{p-1}+1$. Recall $\beta=(\beta',\underbrace{1,\ldots, 1}_{s-j_{p-1}-1})$. Then $\barDel_\calI=\barDel^\epsilon$. Here $\epsilon$ is modified from $\beta$ by contracting an adjacent pair of $1$'s at the $j_p$-th and $(j_p+1)$-th places into a $2$. In any case $\dim\epsilon = \dim\beta - 1 = s-p$. This proves the induction step for this cases and completes the whole proof.
\end{proof}

\begin{corollary} \label{cor: prep for E1 collapse}
Let $\calI\subset \{(2,1,\ldots, 1),\ldots, (1,1\ldots, 2)\}$ where the multi-indices are of length $s$. For each $j=1,\ldots, \# \calI$, the inclusion map $\barDel_\calI\into \barDel_{\partial_j \calI}$ is homologous to zero.
\end{corollary}

Recall that if $\calI=\{\gamma^{(1)},\ldots, \gamma^{(k)}\}$, then $\barDel_{\partial_j \calI}=\bigcap_{i\neq j} \barDel^{\gamma^{(i)}}$ is the intersection omitting the $j$th term.

\begin{proof}
Proposition \ref{prop: intersections as smash prod} shows that there exists the multi-indices $\alpha$ and $\beta$ of length $s$ such that $\barDel_\calI=\barDel^\alpha$ and $\barDel_{\partial_j \calI}=\barDel^\beta$. Since $\barDel_\calI$ is a proper subset of $\barDel_{\partial_j \calI}$, then Proposition \ref{prop: inc are homologous to zero} shows that the inclusion $\barDel_\calI\into \barDel_{\partial_j \calI}$ is homologous to zero, as required.
\end{proof}

\begin{lemma} \label{lem: MV collapses}
If the reduced diagonal map of $A$ is homologous to zero, then the Mayer-Vietoris spectral sequence of $\tilDel_s=\barDel^{(2,1,\ldots, 1)}\union \cdots\union \barDel^{(1,1\ldots, 2)}$ collapses at the $E^1$ term so that
$$\redH_t(\tilDel_s)\cong\Dirsum_{
\substack{
\# \calI+q-1=t \\
\# \calI \ge 1
}
}
\redH_q(\barDel_\calI).$$
Here $\calI$ ranges over the nonempty subsets of $\{(2,1,\ldots, 1),\ldots, (1,1\ldots, 2)\}$.
\end{lemma}

\begin{proof} The differential of the $E^1$-term is given as the following composition:
$$d^1_{p,q}:\co E^1_{p,q}\xrightarrow{j} H_q(F_{p-1})\xrightarrow{i} E^1_{p-1,q}.$$
The homology class of $\alpha_q^{\calI}$ in $\redH_q(\barDel_\calI)$ is mapped to the homology class of $$\sum_{j=1}^{\# \calI} (-1)^{j} \alpha_q^{\partial_j \calI}$$ in $\Dirsum_{j=1}^{\# \calI} \redH_q(\barDel_{\partial_j \calI})$. Since the reduced diagonal is homologous to zero, Corollary \ref{cor: prep for E1 collapse} tells us that each map $\redH_q(\barDel_\calI)\to \redH_q(\barDel_{\partial_j \calI})$ is zero. Therefore $\alpha_q^{\partial_j \calI}=0$ and $\sum_{j=1}^{\# \calI} (-1)^{j} \alpha_q^{\partial_j \calI}=0$ so that the differential $d^1$ is the zero map. Therefore the Mayer-Vietoris spectral sequence collapses at the $E^1$ term.

Using the expression \eqref{eq: generic E1 term} for the $E^1$ term,
$$\redH_t(\tilDel_s)\cong\Dirsum_{p+q-1=t}\,\Dirsum_{\substack{\barDel_\calI\neq \emptyset \\ \# \calI=p\ge 1}} \redH_q(\barDel_\calI)\cong\Dirsum_{
\substack{
\# \calI+q-1=t \\
\# \calI \ge 1
}
}
\redH_q(\barDel_\calI),$$
since no $\barDel_\calI$ is empty.
\end{proof}

\begin{proof}[Proof of Theorem \ref{thm: homology decomp of tilDel}]
By the above Lemma implies the following expression for the mod 2 Betti numbers:
$$\redh_t(\tilDel_s)=\sum_{
\substack{
\# \calI+q-1=t \\
\# \calI \ge 1
}
}
\redh_q(\barDel_\calI).$$
To simplify this expression, recall Proposition \ref{prop: intersections as smash prod} which states that for $p=1,\ldots, s-1$, the collections $\{\barDel_\calI|\, \# \calI=p\}$ and $\{\barDel^\alpha|\, \dim\alpha = s-p\}$ are identical.
Thus the above expression becomes:
\begin{equation} \label{eq: temp of part 2}
\redh_t(\tilDel_s)=
\sum_{\substack{|\alpha|=s \\ (s-\dim\alpha) +q-1=t \\ \dim \alpha\le s-1}} \redh_q(\barDel^\alpha)
= \sum_{\substack{|\alpha|=s \\ q-\dim\alpha = t-s+1 \\ \dim \alpha\le s-1}} \redh_q(\barDel^\alpha).
\end{equation}

Notice if $\dim\alpha=d$, then $$\redh_q(\barDel^\alpha) = \sum_{|\nu|=q} \redh_{\nu_1}(\barDel^{\alpha_1}(A)) \cdots \redh_{\nu_d}(\barDel^{\alpha_d}(A)).$$
Since $\barDel^1(A)=X/G$ by convention and $\barDel^k(A)$ is isomorphic to $A$ for $k=2,3,\ldots$, the homology of $\tilDel_s$ depends only on the homology of $A$ and $X/G$. There must exists constants $c_{\lambda,\mu}$ depending on the multi-indices $\lambda$ and $\mu$ such that
$$\redh_t(\tilDel_s)= \sum_{\lambda,\mu} c_{\lambda,\mu}\redh_\lambda(X/G) \redh_\mu (A).$$

Let $I$ denote $\dim\lambda$ and $J$ denote $\dim\mu$. Thus $c_{\lambda,\mu}$ is the number of multi-indices $\alpha$ which are permutations of $(\underbrace{1,\ldots, 1}_I, a_1,\ldots, a_J)$ for some integers $a_1,\ldots, a_J\ge 2$ that satisfy $I+a_1+\cdots + a_J=s$. After making the substitution $b_i=a_i-2$, this condition is equivalent to $b_1+\cdots + b_J=s-I-2J$ where each $b_i$ is a nonnegative integer. There are $\binom{(s-I-2J)+(J-1)}{J-1}=\binom{s-I-J-1}{J-1}$ nonnegative integer solutions $(b_1,\ldots, b_J)$ to this equation. Thus $c_{\lambda,\mu}=\binom{I+J}{J}\binom{s-I-J-1}{J-1}$.

Since $q=|\nu|=|\lambda|+|\mu|$ and $\dim\nu=\dim\lambda+\dim\mu$, so the condition $q-\dim\alpha=t-s+1$ in \eqref{eq: temp of part 2} is equivalent to $|\lambda|+|\mu|=t-s+\dim\lambda+\dim\mu+1$. Similarly, since $\dim\alpha=\dim\nu$, the condition $\dim\alpha\le s-1$ in \eqref{eq: temp of part 2} is equivalent to $s\ge \dim\lambda+\dim\mu + 1$. Thus we obtain the required
\begin{equation} \label{eq: homology decomp of tilDel}
\redh_t(\tilDel_s;\F_2)\isom \sum_{\substack{|\lambda|+|\mu| =t-s+\dim\lambda+\dim\mu + 1 \\ 2\le\dim\lambda+\dim\mu+1\le s}} c_{\lambda,\mu}\redh_\lambda(X/G;\F_2) \redh_\mu (A;\F_2),
\end{equation}
\end{proof}

Note that \eqref{eq: homology decomp of tilDel} is in fact a finite sum, since the condition $\dim\lambda+\dim\mu+1\le s$ implies that $|\lambda|+|\mu|=t-s+(\dim\lambda+\dim\mu+1)\le t - s + s =t$. As the length is bounded above, there can only be finitely many $\lambda$ and $\mu$ that satisfy $\dim\lambda+\dim\mu+1\le s$.

\section{Proof of Proposition \ref{prop: S2 S1 example}} \label{sec: eg}

We illustrate the efficacy of the homology decompositions in Theorems \ref{thm: homology decomp of JGX} and \ref{thm: homology decomp of tilDel} by computing all the mod 2 Betti numbers of $\Omega(X\rtimes_G E_\infty^1 G)$ for an example $X=S^2\union_{S^1} S^2$. The discrete group $G=C_2$ acts on the 2-sphere $S^2$ antipodally with the equatorial circle $S^1$ as the fixed set. The $G$-space $X$ is formed by taking two 2-spheres $S^2$ with the antipodal action and identifying their equatorial circles.

This pointed $G$-space is equivariantly homotopy equivalent to the following. Take two pairs of discs (that is, four discs in total), and identify all the boundary circles. Let $G$ act on this union $D^2\union D^2\union D^2\union D^2$ by switching the discs in each pair.

\begin{proof}[Proof of Proposition \ref{prop: S2 S1 example}]
Put a simplicial $G$-structure on the $G$-space. Write the simplicial $G$-set as $X=S_1^2\union_{S^1} S_2^2$. The subscripts serve to distinguish each of the two $S^2$'s. For $i=1,2$, let $D^+_i$ denote the upper hemisphere of $S_i^2$ and $D^-_i$ the lower hemisphere. The antipodal $G$-action sends each upper hemisphere to the lower hemisphere, so $D^-_i=D^+_i t$. Then
\begin{align*}
X &=(D^+_1\union D^+_2)\union_{S^1} (D^-_1\union D^-_2) \\
&=(D^+_1\union D^+_2)\union_{S^1} (D^+_1t\union D^+_2t)\\
&=(D^+_1\union D^+_2)\union_{S^1} (D^+_1\union D^+_2)t
\end{align*}
By Proposition \ref{prop: form of decomposable Z2}, the orbit projection of $X$ has a section. Thus Theorem \ref{thm: homology decomp of JGX} applies (here and below we suppress the coefficient $\F_2$ in the notation):
\begin{equation} \label{eq: example}
\redH_n(\Omega(X\rtimes_G W_\infty^1 G))=\Dirsum_{s=1}^\infty \redH_n\left((S^2)^{\smash s}/\tilDel_s\right),
\end{equation}
Since $S^1\to S^1\smash S^1\cong S^2$ is mod 2 homologous to zero, Theorem \ref{thm: homology decomp of tilDel} also applies:
\begin{multline*}
\redh_n(\tilDel_s) = \sum_{J\ge 1}\sum_{I=n-s+1}\binom{I+J}{J}\binom{t-2I-J}{J-1} \\
\underbrace{\redh_2(S^2)\cdots \redh_2(S^2)}_{I} \tensor \underbrace{\redh_1(S^1)\cdots\redh_1(S^1)}_{J}.
\end{multline*}
Since $\redh_2(S^2)=\redh_1(S^1)=1$, so the Betti number is
\begin{align}
\redh_n(\tilDel_s) &= \sum_{J\ge 1}\sum_{I=n-s+1}\binom{I+J}{J}\binom{n-2I-J}{J-1} \notag\\
&= \sum_{J\ge 1}\binom{n-s+1+J}{J}\binom{n-2(n-s+1)+J}{J-1} \notag\\
&= \sum_{J\ge 1}\binom{n-s+1+J}{J}\binom{2s-n-J-2}{J-1} \notag\\
&= \sum_{J= 1}^{2s-3}\binom{n-s+1+J}{J}\binom{2s-n-J-2}{J-1}. \label{eq: example for tilDel}
\end{align}
Note that if the binomial coefficient $\binom{2s-n-J-2}{J-1}$ is nonzero, then $2s-n-J-2\ge J-1$. That is, $n\le 2s-2J-1\le 2s-3$ since $J\ge 1$. Thus $\redH_n(\tilDel_s)=0$ if $n>2s-3$. Combining this observation with the fact that the only nontrivial homology group of $((S^2\union_{S^1} S^2t)/G)^{\smash s}=(S^2)^{\smash s}=S^{2s}$ is in the $2s$-th dimension, the short exact sequence $\tilDel_s\to S^{2s} \to S^{2s}/\tilDel_s$ induces the following long exact sequence in homology:
\begin{center}
\begin{tabular}{cccccc}
        &   $\cdots$                   & $\to$ & 0                         & $\to$ & $\redH_{2s+2}\left(S^{2s}/\tilDel_s\right)$\\
  $\to$ & 0                            & $\to$ & 0                         & $\to$ & $\redH_{2s+1}\left(S^{2s}/\tilDel_s\right)$ \\
  $\to$ & 0                            & $\to$ & $\redH_{2s}(S^{2s})=\F_2$ & $\to$ & $\redH_{2s}\left(S^{2s}/\tilDel_s\right)$ \\
  $\to$ & 0                            & $\to$ & 0                         & $\to$ & $\redH_{2s-1}\left(S^{2s}/\tilDel_s\right)$ \\
  $\to$ & 0                            & $\to$ & 0                         & $\to$ & $\redH_{2s-2}\left(S^{2s}/\tilDel_s\right)$ \\
  $\to$ & $\redH_{2s-3}(\tilDel_{s})$ & $\to$ & 0                         & $\to$ & $\redH_{2s-3}\left(S^{2s}/\tilDel_s\right)$ \\
        &                              &       & $\cdots$                  &       &  \\
  $\to$ & $\redH_{1}(\tilDel_{s})$     & $\to$ & 0                         & $\to$ & $\redH_{1}\left(S^{2s}/\tilDel_s\right)$ \\
  $\to$ & 0                            &       &                   &       &  \\
\end{tabular}
\end{center}
Thus
$$\redH_n\left(S^{2s}/\tilDel_s\right)=\begin{cases}
0,                      & n\ge 2s+1,\\
\F_2                    & n=2s, \\
0,                      & n=2s-1, \\
\redH_{n-1}(\tilDel_s)  & n\le 2s-2.
\end{cases}$$
For $k\ge 1$, applying this formula to \eqref{eq: example} gives
\begin{align*}
\redH_{2k}(\Omega(X\rtimes_G W_\infty^1 G)) &\isom \redH_{2k}\left(S^{2k}/\tilDel_k\right)\dirsum\Dirsum_{r=k+1}^\infty \redH_{2k}\left(S^{2r}/\tilDel_r\right) \\
&\isom \F_2 \dirsum \Dirsum_{r=k+1}^\infty \redH_{2k-1}(\tilDel_r). \\
\end{align*}
By \eqref{eq: example for tilDel}, the even Betti number is:
\begin{align*}
\redh_{2k}(\Omega(X\rtimes_G W_\infty^1 G)) &= 1+\sum_{r=k+1}^\infty \sum_{J= 1}^{2r-3}\binom{2k-r+J}{J}\binom{2r-2k-J-1}{J-1} \\
&= 1+\sum_{r=k+1}^{2k} \sum_{J= 1}^{2r-3}\binom{2k-r+J}{J}\binom{2r-2k-J-1}{J-1}.
\end{align*}
Here the upper bound $r\le 2k$ is obtained by observing that $\binom{2k-r+J}{J}$ is nonzero only if $2k-r+J\ge J$ or $r\le 2k$.

Similarly we can compute the odd Betti number:
$$\redh_{2k+1}(\Omega(X\rtimes_G W_\infty^1 G)) = \sum_{r=k+2}^{2k+1} \sum_{J= 1}^{r-k-1}\binom{2k-r+J+1}{J}\binom{2r-2k-J-2}{J-1}$$
for $k\ge 0$.

Take geometric realization to obtain the required result.
\end{proof}

Using these formulas, we compute by hand the Betti numbers in the dimension 1 to 12 to be
$$\{\redh_n(\Omega(X\rtimes_G E_\infty^1 G);\F_2)\}_{n=1,\ldots, 12}=\{0,2,1,5,5,14,19,42,66,131,221,417\}.$$
A search with the Online Encyclopedia of Integer sequences \cite{OEIS} gives the sequence A052547. For $n\ge 0$, set $a_n$ to be the coefficient of $x^n$ in the power series expansion of $(1-x)/(x^3 - 2x^2 - x+1)$. The Encyclopedia informs us that, for $1\leq n\leq 12$:
\[
a_n= \redh_n(\Omega(X\rtimes_G E_\infty^1 G);\F_2)
\]
Note that for $n=0$, the initial term $a_0=1$ of sequence A052547 differs from $\redh_0(\Omega(X\rtimes_G E_\infty^1 G);\F_2)=0$; this is because we are using the reduced homology. This leads us to conjecture the following:
\begin{conjecture}
The reduced mod $2$ Poincar\'{e} series of $\Omega(X\rtimes_G E_\infty^1 G)$ is
\[
\sum_{n=0}^\infty \redh_n(\Omega(X\rtimes_G E_\infty^1 G);\F_2) \,x^n = \frac{1-x}{x^3 - 2x^2 - x+1} - 1.
\]
\end{conjecture}

The sequence $a_n$ has a geometric interpretation in terms diagonals lengths in the regular heptagon with unit side length (see \cite{Steinbach97,Lang12}.)
These diagonal lengths are related to the Chebyshev polynomials, which are important in approximation theory.

\end{document}